\documentclass[11pt,twoside,reqno]{amsart}

\usepackage{amssymb, latexsym}
\usepackage[italian,english]{babel}
\usepackage{amsmath,amsfonts,amsthm}
\usepackage{paralist}
\usepackage[top=2.50cm, bottom=2.50cm, left=2.50cm, right=2.50cm]{geometry}

\usepackage{color}

\numberwithin{equation}{section}

\usepackage{xspace}
\usepackage[bookmarksnumbered,colorlinks]{hyperref}
\usepackage{graphics}
\def\bb#1\eb{\textcolor{blue}
{#1}} %
\def\br#1\er{\textcolor{red}
{#1}} %
\def\bv#1\ev{\textcolor{green}
{#1}} %
\def\bc#1\ec{\textcolor{cyan}
{#1}} %

\usepackage{graphics}

\def\Xint#1{\mathchoice
  {\XXint\displaystyle\textstyle{#1}}%
  {\XXint\textstyle\scriptstyle{#1}}%
  {\XXint\scriptstyle\scriptscriptstyle{#1}}%
  {\XXint\scriptscriptstyle\scriptscriptstyle{#1}}%
  \!\int}
\def\XXint#1#2#3{{\setbox0=\hbox{$#1{#2#3}{\int}$}
  \vcenter{\hbox{$#2#3$}}\kern-.5\wd0}}
\def\-int{\Xint -}

\newcommand{\h}{\mathbb{H}^{s}_{2\pi}}
\newcommand{\R}{\mathbb{R}}
\newcommand{\N}{\mathbb{N}}
\newcommand{\C}{\mathcal{C}}
\newcommand{\Z}{\mathbb{Z}}
\newcommand{\B}{(-\pi, \pi)^{N}}
\newcommand{\T}{\mathcal{S}_{2\pi}}
\newcommand{\X}{H^{1}_{0, L}(\mathcal{C})}
\newcommand{\Y}{\mathbb{X}^{s}_{2\pi}}
\newcommand{\A}{\mathbb{H}^{s}(\Omega)}
\newcommand{\D}{\textup{Dir}}

\DeclareMathOperator{\dive}{div}

\DeclareMathOperator{\e}{\varepsilon}

\newtheorem{prop}{Proposition}[section]
\newtheorem{lem}{Lemma}[section]
\newtheorem{thm}{Theorem}[section]

\newtheorem{remark}{Remark}[section]

\begin{document}
\title[Fractional Ambrosetti-Prodi problems]{An Ambrosetti-Prodi type result for fractional spectral problems} 

\author[V. Ambrosio]{Vincenzo Ambrosio}
\address{Vincenzo Ambrosio\hfill\break\indent 
Department of Mathematics  \hfill\break\indent
EPFL SB CAMA \hfill\break\indent
Station 8 CH-1015 Lausanne, Switzerland}
\email{vincenzo.ambrosio2@unina.it}

\keywords{Fractional spectral operators; variational methods; Ambrosetti-Prodi problem}
\subjclass[2010]{35A15, 35J60, 35R11, 45G05}

\begin{abstract}
We consider the following class of fractional parametric problems
\begin{equation*}
\left\{
\begin{array}{ll}
(-\Delta_{\D})^{s} u= f(x, u)+t\varphi_{1}+h &\mbox{ in } \Omega\\
u=0 &\mbox{ on } \partial \Omega,
\end{array}
\right.
\end{equation*}
where $\Omega\subset \R^{N}$ is a smooth bounded domain, $s\in (0, 1)$, $N> 2s$, $(-\Delta_{\D})^{s}$ is the fractional Dirichlet Laplacian, $f: \bar{\Omega} \times \R\rightarrow \R$ is a locally Lipschitz nonlinearity having linear or superlinear growth and satisfying Ambrosetti-Prodi type assumptions, $t\in \R$, $\varphi_{1}$ is the first eigenfunction of the Laplacian with homogenous boundary conditions, and $h:\Omega\rightarrow \R$ is a bounded function. Using variational methods, we prove that there exists a $t_{0}\in \R$ such that the above problem admits at least two distinct solutions
for any $t\leq t_{0}$. 
We also discuss the existence of solutions for a fractional periodic Ambrosetti-Prodi type problem.
\end{abstract}

\maketitle
\section{Introduction}

\noindent
In this paper we deal with the multiplicity of solutions for the following nonlinear fractional equation 
\begin{equation}\label{P0}
\left\{
\begin{array}{ll}
(-\Delta_{\D})^{s} u= f(x, u)+t\varphi_{1}+h &\mbox{ in } \Omega\\
u=0 &\mbox{ on } \partial \Omega,
\end{array}
\right.
\end{equation}
where $\Omega$ is a smooth bounded domain of $\R^{N}$, $N> 2s$, $t\in \R$ is a parameter, 
$h\in L^{\infty}(\Omega)$ is such that $\int_{\Omega} h \varphi_{1} \, dx=0$ 
and $f:\bar{\Omega}\times \R\rightarrow \R$ is a locally Lipschitz function verifying
\begin{equation}\label{f1}
\limsup_{t\rightarrow -\infty} \frac{f(x, t)}{t}<\lambda_{1}^{s}< \liminf_{t\rightarrow \infty} \frac{f(x, t)}{t} \quad \mbox{ uniformly for a.e. } x\in \Omega.
\end{equation}
A typical example for $f$ is given by the jumping nonlinearity $f(x, t)=\gamma_{1} t^{+}-\gamma_{2} t^{-}$ with $\gamma_{2}<\lambda_{1}^{s}<\gamma_{1}$.
The operator $(-\Delta_{\D})^{s}$ is defined as follows. Let us denote by $(\varphi_{k}, \lambda_{k})$ the eigenvalues and corresponding eigenfunctions of the Laplacian $-\Delta$ in $\Omega$ with Dirichlet boundary conditions
\begin{equation*}
\left\{
\begin{array}{ll}
-\Delta \varphi_{k}=\lambda_{k} \varphi_{k} &\mbox{ in } \Omega\\
\varphi_{k}=0 &\mbox{ on } \partial \Omega,
\end{array}
\right.
\end{equation*}
normalized by $|\varphi_{k}|_{2}=1$ for any $k\in \N$.
Then, the fractional Dirichlet Laplacian is given by 
$$
(-\Delta_{\D})^{s} u(x)=\sum_{k=1}^{\infty} c_{k} \lambda_{k}^{s} \varphi_{k}(x)  \quad \forall x\in \Omega
$$
for any $u=\sum_{k=1}^{\infty} c_{k} \varphi_{k}\in C^{\infty}_{c}(\Omega)$ where $c_{k}=\int_{\Omega} u \varphi_{k} \, dx$.
This definition can be extended by density for any function $u$ belonging to the Hilbert space
$$
\A=\Bigl\{u\in L^{2}(\Omega): \sum_{k=1}^{\infty} |c_{k}|^{2}\lambda_{k}^{s}<\infty\Bigr\}
$$
which can be characterized \cite{LM} as follows
\begin{equation*}
\A=
\left\{
\begin{array}{ll}
H^{s}(\Omega) &\mbox{ if } s\in (0, \frac{1}{2})\\
H^{\frac{1}{2}}_{00}(\Omega) &\mbox{ if } s=\frac{1}{2} \\
H^{s}_{0}(\Omega) &\mbox{ if } s\in (\frac{1}{2}, 1).
\end{array}
\right.
\end{equation*}
The local counterpart of problem \eqref{P0}  corresponds to the celebrated boundary value problem 
\begin{equation}\label{APP}
\left\{
\begin{array}{ll}
-\Delta u= f(x, u)+t\varphi_{1}+h &\mbox{ in } \Omega\\
u=0 &\mbox{ on } \partial \Omega
\end{array}
\right.
\end{equation}
studied by Ambrosetti and Prodi \cite{AP} in $1972$. The authors in \cite{AP} proved that if  $f(x, u)=f(u)$ is a $C^{2}$ strictly convex function verifying $f(0)=0$, $0<f'(-\infty)<\lambda_{1}<f'(\infty)<\lambda_{2}$ and $g(x)=t \varphi_{1}+h$ with $h\in C^{0, \alpha}(\bar{\Omega})$, then there exists a closed connected $C^{1}$ manifold $M$ of codimension $1$ in the space $C^{0, \alpha}(\bar{\Omega})$ which splits the space into two connected components $S_{0}$ and $S_{2}$ such that, if $g\in S_{0}$ then \eqref{APP} has no solution, if $g\in M$ then \eqref{APP} has exactly one solution, and if $g\in S_{2}$ then \eqref{APP} has exactly two solutions. Their proof is obtained by using results on differentiable mappings with singularities.  
Subsequently, different variants and formulations of this problem have been extensively studied by many authors.
For more details, we refer to \cite{AH, ARu, BP,  Chang, Dancer, DF1, DF2, DFS, Hess, KW} where several existence and multiplicity results have been obtained for various Ambrosetti-Prodi type problems by exploiting suitable and various topological and variational methods.


Recently there has been an increasing interest in the study of nonlinear partial differential equations involving fractional and nonlocal operators of elliptic type. These operators arise in a great number of applications such as phase transitions, image processing, mathematical finance, optimization, quantum mechanics and obstacle problem.
The literature on fractional and nonlocal operators and on their applications is so huge that we do not even try to collect here a detailed bibliography. Anyway, we refer the interested reader to \cite{DPV, MBRS} for an elementary introduction on this subject. \\
Motivated by the interest shared by the mathematical community in this topic, in the present paper, we consider a fractional version of the Ambrosetti-Prodi problem involving the fractional Dirichlet  Laplacian. In particular, we are interested in finding solutions for \eqref{P0} when $f$ has linear growth, that is
\begin{equation}\label{f2}
|f(x, t)|\leq c_{1}|t|+c_{2}  \mbox{ for any } x\in \bar{\Omega}, t\in \R,
\end{equation}
for some $c_{1}, c_{2}>0$, 
or when $f$ satisfies the following conditions: 
\begin{equation}\label{f3}
|f(x, t)|\leq c_{3}|t|^{p-1}+c_{4}  \mbox{ for any } x\in \bar{\Omega}, t\in \R,
\end{equation}
with $c_{3}, c_{4}>0$ are constants and  $2\leq p<2^{*}_{s}:=\frac{2N}{N-2s}$,
and $f$ verifies the Ambrosetti-Rabinowitz condition \cite{AR}, that is
\begin{equation}\label{f4}
\mbox{ there exist } \theta>2, r>0  \mbox{ such that } 
0<\theta F(x, t)\leq t f(x, t) \mbox{ for any } x\in \Omega, t\geq r,
\end{equation}
where  $F(x, \tau)=\int_{0}^{\tau} f(x, \xi) d\xi$.\\
Our main result can be stated as follows:
\begin{thm}\label{thm1}
Assume that $f$ satisfies either \eqref{f1}-\eqref{f3} or \eqref{f1}-\eqref{f3}-\eqref{f4}.
Then there exists $t_{0}\in \R$ such that for any $t\leq t_{0}$, the problem \eqref{P0} admits at least two solutions in $C^{0, \alpha}(\bar{\Omega})$, for some $\alpha\in (0,1)$. 
\end{thm}

The proof of the Theorem \ref{thm1} is obtained by applying critical point theory after transforming \eqref{P0} into an elliptic nonlinear Neumann problem in a half-cylinder via a suitable variant of the extension method introduced by Caffarelli and Silvestre in \cite{CS}. This approach has been brilliantly used by several authors to study existence, multiplicity, regularity and symmetry properties of solutions for different fractional problems (in $\R^{N}$ or in bounded domains) involving subcritical and critical nonlinearities; see  for instance \cite{AM, A22, A30, BCDPS, BrCDPS, CT, DMV, YYY}. 
Taking into account this fact, instead of \eqref{P0} we are lead to consider the following degenerate elliptic problem with a nonlinear Neumann boundary condition
\begin{equation}\label{P}
\left\{
\begin{array}{ll}
-\dive(y^{1-2s} \nabla U)= 0 &\mbox{ in } \C:=\Omega\times (0, \infty)\\
U=0 &\mbox{ on } \partial_{L} \C:=\partial \Omega\times [0, \infty) \\
\frac{\partial U}{\partial \nu^{1-2s}}=\kappa_{s}[f(x, u)+t \varphi_{1}+h] &\mbox{ on } \partial^{0} \C:=\Omega\times \{0\},
\end{array}
\right.
\end{equation}
where $\kappa_{s}$ is a suitable positive constant (see \cite{BCDPS, BrCDPS, CT, CDDS, ST}) and $u$ is the trace of $U$ (see Section $2$). For simplicity, we will assume that $\kappa_{s}=1$.\\
Due to the variational structure of the problem \eqref{P}, we introduce the following Euler-Lagrange functional $\Phi: \X\rightarrow \R$ associated to \eqref{P}, that is
\begin{equation}\label{phi}
\Phi(U)=\frac{1}{2} \|U\|^{2}-\int_{\Omega} F(x, u) \,dx-\int_{\Omega} (t \varphi_{1}+h)u \,dx,
\end{equation}
where $\X$ is defined as 
$$
H^{1}_{0,L}=\overline{C_{c}^{\infty}(\Omega\times[0, \infty))}^{\|\cdot\|}
$$
endowed with the norm
$$
\|U\|^{2}=\iint_{\mathcal{C}} y^{1-2s} |\nabla U|^{2} \, dx dy;
$$
see Section $2$ below and \cite{BCDPS, BrCDPS, CT, CDDS} for more details.\\
Then it is clear that solutions to \eqref{P} can be obtained as critical points of $\Phi$.
We recall that  $U\in \X$ is a weak solution to \eqref{P} if for any $\Psi\in \X$ it holds
$$
\iint_{\C} y^{1-2s}\nabla U \nabla \Psi \, dx dy= \int_{\Omega} [g(x, u)+t \varphi_{1}+h] \psi \, dx. 
$$
Borrowing some ideas used in \cite{DFS}, we show that \eqref{P} admits at least two solutions. 
Firstly, we prove the existence of a subsolution for any $t\in \R$, and a supersolution for any $t$ sufficiently small. After that, we apply some abstract critical point results \cite{DFS} to deduce the existence of a local minimizer of $\Phi$ in $\X$. This will be done looking for critical points  of $\Phi$ restricted to a suitable closed convex subset of $\X$. Let us point out that in the the proofs that we perform, we have to take care of the trace term, so a more careful analysis  is needed with respect to the classic case in order to overcome the difficulties coming from the non-local character of $(-\Delta_{\D})^{s}$. 
Then, applying a variant of the  Mountain Pass Theorem \cite{AR}, we deduce the existence of a second solution to \eqref{P} different from the previous one. We also note that it is possible to obtain a nonexistence result to \eqref{P} for $t$ large (see Remark \ref{REM3.1}).\\
In the second part of this work, we consider a periodic version of the problem \eqref{P0}, namely 
\begin{equation}\label{Pp}
\left\{
\begin{array}{ll}
(-\Delta+m^{2})^{s} u = f(x, u)+t+h \, &\mbox{ in } \, (-\pi, \pi)^{N},\\
u(x+2\pi e_{i})=u(x) \, &\mbox{ for all } \, x\in\R^{N}, i=1, \dots, N
\end{array}
\right.
\end{equation}
where $m>0$, $s\in (0, 1)$, $N>2s$, $h\in L^{\infty}((-\pi, \pi)^{N})$ verifies $\int_{(-\pi, \pi)^{N}} h \, dx=0$ and $f: \R^{N}\times \R \rightarrow \R$ is a locally Lipschitz function, $2\pi$-periodic in $x$, such that
\begin{equation}\label{f11}
\limsup_{t\rightarrow -\infty} \frac{f(x, t)}{t}<m^{2s}<\liminf_{t\rightarrow \infty} \frac{f(x, t)}{t} \quad \mbox{ uniformly for a.e. } x\in (-\pi, \pi)^{N},
\end{equation}
and satisfying either \eqref{f2} or \eqref{f3}-\eqref{f4}.
Here, the fractional spectral operator $(-\Delta+m^{2})^{s}$ is a non-local operator which can be defined for any $u=\sum_{k\in \Z^{N}} c_{k} \frac{e^{\imath k\cdot x}}{(2\pi)^{\frac{N}{2}}}\in \mathcal{C}^{\infty}_{2\pi}(\R^{N})$, that is $u$ is infinitely differentiable in $\R^{N}$ and $2\pi$-periodic in each variable, by setting
\begin{equation}\label{nfrls}
(-\Delta+m^{2})^{s} u(x)=\sum_{k\in \Z^{N}} c_{k} (|k|^{2}+m^{2})^{s} \, \frac{e^{\imath k\cdot x}}{(2\pi)^{\frac{N}{2}}}
\end{equation}
where
$c_{k}:=\frac{1}{(2\pi)^{\frac{N}{2}}} \int_{(-\pi,\pi)^{N}} u(x)e^{- \imath k \cdot x}dx \quad (k\in \Z^{N})$
are the Fourier coefficients of the function $u$.\\
This operator can be extended by density on the Hilbert space
$$
\h:=\left\{u=\sum_{k\in \Z^{N}} c_{k} \frac{e^{\imath k\cdot x}}{(2\pi)^{\frac{N}{2}}}\in L^{2}(-\pi,\pi)^{N}: \sum_{k\in \Z^{N}} (|k|^{2}+m^{2})^{s} \, |c_{k}|^{2}<+\infty \right\}
$$
endowed with the norm
$$
|u|_{\h}:=\left(\sum_{k\in \Z^{N}} (|k|^{2}+m^{2})^{s} |c_{k}|^{2}\right)^{1/2}.
$$
Then, using the fact that the extension technique works again in periodic setting \cite{A1, A2, RS, ST}, and following the ideas above discussed, we are able to prove the next result. 
\begin{thm}\label{thm2}
Assume that $f$ satisfies either \eqref{f11}-\eqref{f3} or \eqref{f11}-\eqref{f3}-\eqref{f4}.
Then there exists $t_{0}\in \R$ such that for any $t\leq t_{0}$, the problem \eqref{Pp} admits at least two solutions in $C^{0, \alpha}([-\pi, \pi]^{N})$, for some $\alpha\in (0,1)$. 
\end{thm}
\noindent
As far as we know, these are the first results for Ambrosetti-Prodi type problems in the fractional  framework.
The paper is organized as follows. In Section $2$ we give some results about the involved fractional Sobolev spaces, and we recall some useful critical point results. In Section $3$ we prove the existence of sub and supersolutions to \eqref{P} and we show that the functional $\Phi$ satisfies the Palais-Smale condition. In Section $4$ we give the proof of Theorem \ref{thm1}. In the last Section we discuss a fractional periodic version of the Ambrosetti-Prodi problem.

\section{preliminaries}
\noindent
In this section we collect some useful notations and basic results which will be useful along the paper. For more details we refer to \cite{BrCDPS, CT, CDDS}.\\
We use the notation $|\cdot|_{q}$ to denote the $L^{q}(\Omega)$-norm of any measurable function $u:\Omega\rightarrow \R$.\\
Fix $s\in (0, 1)$.
We say that $u\in H^{s}(\Omega)$ if $u\in L^{2}(\Omega)$ and it holds
$$
[u]^{2}=\iint_{\Omega\times \Omega} \frac{|u(x)-u(y)|^{2}}{|x-y|^{N+2s}} \, dx dy<\infty.
$$
We define $H^{s}_{0}(\Omega)$ as the closure of $C^{\infty}_{c}(\Omega)$ with respect to the norm $[u]^{2}+|u|_{2}^{2}$. The space $H^{\frac{1}{2}}_{00}(\Omega)$ is the Lions-Magenes space  which consists of the function $u\in H^{\frac{1}{2}}(\Omega)$ such that
$$
\int_{\Omega} \frac{u^{2}(x)}{dist(x, \partial \Omega)}\, dx<\infty.
$$
Let us introduce the following Hilbert space
$$
\A=\Bigl\{u\in L^{2}(\Omega): \sum_{k=1}^{\infty} |c_{k}|^{2}\lambda_{k}^{s}<\infty\Bigr\}.
$$
It is well-known \cite{LM} that by interpolation it results 
\begin{equation*}
\A=
\left\{
\begin{array}{ll}
H^{s}(\Omega) &\mbox{ if } s\in (0, \frac{1}{2})\\
H^{\frac{1}{2}}_{00}(\Omega) &\mbox{ if } s=\frac{1}{2} \\
H^{s}_{0}(\Omega) &\mbox{ if } s\in (\frac{1}{2}, 1).
\end{array}
\right.
\end{equation*}
Now, let $D\subset \R^{N+1}$ be an open set. We say that a measurable function $U: D \rightarrow \R$ belongs to the weighthed space $L^{2}(D, |y|^{1-2s})$ if $U$ verifies the following condition
$$
\iint_{D} |y|^{1-2s} U^{2}(x,y)\, dx dy<\infty.
$$
We say that $U\in H^{1}(D, |y|^{1-2s})$ if $U\in L^{2}(D, |y|^{1-2s})$ and its weak derivatives $\nabla U$ exist and belong to $L^{2}(D, |y|^{1-2s})$. It is clear that $H^{1}(D, |y|^{1-2s})$ is a Hilbert space with the inner product 
$$
\langle U, V\rangle=\iint_{D} |y|^{1-2s} (\nabla U \nabla V+UV)\, dx dy,
$$
and $C^{\infty}(D)\cap H^{1}(D, |y|^{1-2s})$ is dense in $H^{1}(D, |y|^{1-2s})$ (see for instance Proposition $2.1.2$ and Corollary $2.1.6$ in \cite{T}). \\
Let $\Omega\subset \R^{N}$ be a smooth bounded domain. Let us define $\C=\Omega\times (0, \infty)$ the half-cylinder of basis $\Omega$, and we denote by $\partial_{L}\C=\partial \Omega\times [0, \infty)$ its lateral boundary and we set $\partial^{0}\C=\Omega\times \{0\}$. \\ 
We define $\X$ as the completion of the functions $U\in C^{\infty}_{c}(\Omega\times [0, \infty))$ with respect to the norm
$$
\|U\|^{2}=\iint_{\C} y^{1-2s} |\nabla U|^{2} \, dx dy.
$$
Now, we recall the following trace theorem which relates $\X$ to $\A$.
\begin{thm}\cite{BrCDPS, CT, CDDS}\label{tracethm}
There exists a surjective continuous linear map $\textup{Tr}: \X\rightarrow \A$ such that, for any $U\in \X$
$$
\kappa_{s} |\textup{Tr}(U)|^{2}_{\A}\leq \|U\|^{2}.
$$
\end{thm}
\noindent
We also have some useful embedding results.
\begin{thm}\cite{BrCDPS, CT, CDDS}\label{Sembedding}
Let $N> 2s$ and $q\in [1, 2^{*}_{s}]$. Then there exists a constant $C$ depending on $N$, $q$ and the measure of $\Omega$, such that, for all $U\in \X$
$$
|u|_{q}\leq C \|U\|.
$$
Moreover, $\X$ is compactly embedded into $L^{\nu}(\Omega)$ for any $\nu\in [1, 2^{*}_{s})$.
\end{thm}

\noindent
Thus, we get the following fundamental result which allows us to realize the fractional Dirichlet operator.
\begin{thm}\cite{BrCDPS, CT, CDDS}
Let $u\in \A$. Then there exists a unique $U\in \X$ such that 
\begin{equation*}
\left\{
\begin{array}{ll}
-\dive(y^{1-2s} \nabla U)= 0 &\mbox{ in } \C \\
U=0 &\mbox{ on } \partial_{L} \C\\
\textup{Tr}(U)=u &\mbox{ on } \partial^{0} \C.
\end{array}
\right.
\end{equation*}
Moreover
$$
\frac{\partial U}{\partial \nu^{1-2s}}:=-\lim_{y\rightarrow 0^{+}} y^{1-2s} U_{y}(x, y)=\kappa_{s}(-\Delta_{\D})^{s} u(x) \mbox{ in } \A^{*},  
$$
where $\A^{*}$ is the dual of $\A$. The function $U$ is called the extension of $u$.
\end{thm}

\noindent
For simplicity, we will use the notation $u$ instead of $\textup{Tr}(U)$ for any function $U$ in $\X$. 
We also recall that if $u=\sum_{k=1}^{\infty} c_{k} \varphi_{k}\in \A$, the extension of $u$ is given by 
$$
U(x, y)=\sum_{k=1}^{\infty} c_{k} \varphi_{k}(x) \theta(\sqrt{\lambda_{k}} y),
$$
where $\theta\in H^{1}(\R_{+}, y^{1-2s})$ solves the problem 
\begin{equation*}
\left\{
\begin{array}{ll}
\theta''+\frac{1-2s}{y}\theta'-\theta= 0 &\mbox{ in } \R_{+}\\
\theta(0)=1, \mbox{ and } -\lim_{y\rightarrow 0^{+}} y^{1-2s}\theta'(y)=\kappa_{s}.
\end{array}
\right.
\end{equation*}
In addition, $\|U\|^{2}=\kappa_{s}|u|^{2}_{\A}$; see \cite{BrCDPS, CT, CDDS} for more details.

\noindent
Now, we prove the following result concerning the H\"older continuity of solutions to \eqref{P}. 
\begin{lem}\label{boundlem}
Let $U\in \X$ be a weak solution to 
\begin{equation}\label{Pg}
\left\{
\begin{array}{ll}
-\dive(y^{1-2s} \nabla U)= 0 &\mbox{ in } \C\\
U=0 &\mbox{ on } \partial_{L} \C\\
\frac{\partial U}{\partial \nu^{1-2s}}=g(x, u) &\mbox{ on } \partial^{0} \C,
\end{array}
\right.
\end{equation}
where $|g(x, t)|\leq C_{0}(1+|t|^{q})$ in $\Omega\times \R$, for some $1\leq q\leq 2^{*}_{s}-1$. 
Then $u\in L^{\infty}(\Omega)$ and there exists $M\in C(\R_{+})$, depending only on $C_{0}$, $N$, and $\Omega$, such that $|u|_{\infty}\leq M(|u|_{2^{*}_{s}})$. Moreover, $U\in C^{0, \alpha}(\bar{\C})$ and $u\in C^{0, \alpha}(\bar{\Omega})$, for some $\alpha\in (0, 1)$.
\end{lem}
\begin{proof}
For $\beta>0$ and $T\geq 1$, we define $U_{T}=\min\{|U|, T\}$ and we take $UU_{T}^{2\beta}\in \X$ as test function in \eqref{Pg}. Then a direct calculations yields 
\begin{align*}
&\iint_{\C} y^{1-2s} U_{T}^{2\beta}|\nabla U|^{2}dx dy+2\beta\iint_{D_{T}} y^{1-2s}|U|^{2\beta}|\nabla U|^{2}dx dy\\
&=\iint_{\C} y^{1-2s} \nabla U \nabla (UU^{2\beta}_{T}) dx dy\\
&=\int_{\Omega} g(x, u)uu^{2\beta}_{T}dx \\
&\leq \int_{\Omega} a(x)(1+|u|)^{2}u_{T}^{2\beta} dx,
\end{align*}
where $$0\leq a(x):=\frac{|g(x,u)|}{1+|u(x)|}\leq C(1+|u|^{q-1})\in L^{\frac{N}{2s}}(\Omega)$$ and $$D_{T}=\{(x, y)\in \C: |U(x, y)|<T\}.$$

Since 
$$
\iint_{\C} y^{1-2s} |\nabla (UU^{\beta}_{T})|^{2}dx dy=\iint_{\C} y^{1-2s} U_{T}^{2\beta}|\nabla U|^{2}dx dy+(2\beta+\beta^{2})\iint_{D_{T}} y^{1-2s}|U|^{2\beta}|\nabla U|^{2}dx dy,
$$
we can deduce that
$$
\iint_{\C} y^{1-2s} |\nabla (UU^{\beta}_{T})|^{2}dx dy\leq C(\beta+1)\int_{\Omega} a(x)(1+|u|)^{2}u_{T}^{2\beta} dx.
$$
From Theorem \ref{Sembedding} it follows that
$$
|uu_{T}^{\beta}|_{2^{*}_{s}}^{2}\leq C(\beta+1)\int_{\Omega} a(x)(1+|u|)^{2}u_{T}^{2\beta} dx.
$$
At this point we can argue as in the proof of Proposition $5.1$ in \cite{BCDPS} to deduce that $u\in L^{r}(\Omega)$ for any $r\in [2, \infty)$. As a consequence, $g(\cdot, u)\in L^{r}(\Omega)$ for some $r>\frac{N}{2s}$ and applying Theorem $4.7$ in \cite{BCDPS} we deduce that $u\in L^{\infty}(\Omega)$ and $U\in L^{\infty}(\C)$. Therefore $g(\cdot, u)\in L^{\infty}(\Omega)$ and using Corollary $4.8$ in \cite{BCDPS} we get  $U\in C^{0, \alpha}(\bar{\C})$ for some $\alpha\in (0, 1)$.
\end{proof}

\noindent
Now, following the proof of Lemma $2.7$ in \cite{CDDS}, we can establish the following maximum principle:
\begin{lem}\label{MPcdds}
Let $\Omega\subset \R^{N}$ denote any domain and take $R>0$. Let $U$ denote any locally integrable function on $\Omega\times (0, R)$ such that $\nabla U\in L^{2}(\Omega\times (0, R), y^{1-2s})$ and $c\geq 0$.
Assume that $-\dive(y^{1-2s} \nabla U)=0$ in $\Omega\times (0, R)$, $U\geq 0$ in $\Omega\times (0, R)$ and 
$$
\int_{\Omega\times (0, R)} y^{1-2s} \nabla U\nabla \Psi dxdy+\int_{\Omega} c u\psi dx\geq 0
$$
for all $\Psi\in H^{1}(\Omega\times (0, R), y^{1-2s})$ such that $\Psi\geq 0$ a.e. in $\Omega\times (0, R)$ and $\Psi=0$ on $\partial \Omega\times (0, R)\cup \Omega\times \{R\}$. Then, either $U\equiv 0$, or for any compact subset $K$ of $\Omega\times [0, R)$, $ess\inf_{K} U>0$. 
\end{lem}
\begin{proof}
Let $\tilde{U}$ denote the even extension of $U$ with respect to the $y$ variable in $\Omega\times (-R, R)$. Then
$$
\int_{\Omega\times (-R, R)} y^{1-2s} \nabla \tilde{U}\nabla \Psi dxdy+\int_{\Omega} c \tilde{u}\psi dx\geq 0
$$
for any $\Psi\in H^{1}(\Omega\times (-R, R), y^{1-2s})$ such that $\Psi\geq 0$ a.e. in $\Omega\times (-R, R)$ and $\Psi=0$ on $\partial (\Omega\times (-R, R))$. Using Proposition $2$ in \cite{FF}, either $\tilde{U}\equiv 0$, or $ess\inf_{K} U>0$ for any compact subset $K$ of $\Omega\times (-R, R)$.
\end{proof}

\noindent
We conclude this Section recalling some useful abstract results whose proofs can be found in \cite{DFS}:
\begin{thm}\cite{DFS}\label{thm3.1}
Let $X$ be a real Hilbert space and $\Phi\in C^{1}(X, \R)$ be such that $\Phi'$ is of type $S^{+}$, that is if 
$u_{n}\rightharpoonup u$ in  $X$ and $\limsup_{n\rightarrow \infty} \langle \Phi'(u_{n}), u_{n}-u\rangle\leq 0$ then $u_{n}\rightarrow u$ in $X$.
Suppose that $\Phi$ is bounded below in a ball $\bar{B}$. Then there exists $v_{0}\in \bar{B}$ such that $\Phi(v_{0})=\inf_{u\in \bar{B}} \Phi(u)$ and $\Phi'(v_{0})=\lambda v_{0}$ for some $\lambda\leq 0$.
\end{thm}

\begin{prop}\cite{DFS}\label{MP}
Let $X$ be a real Hilbert space and $\Phi\in C^{1}(X, \R)$ be a functional satisfying the Palais-Smale condition.
Assume that
$$
\max\{\Phi(0), \Phi(e)\}\leq \inf\{\Phi(u): \|u\|=r \},
$$
for some $e\in X$, and $0<r<\|e\|$. Then $\Phi$ has a critical point $u_{0}\neq 0$.
\end{prop}

\begin{prop}\cite{DFS}\label{Lm}
Let $\Phi: X\rightarrow \R$ be a $C^{1}$ functional defined in a Hilbert space $X$. Let $K$ be a closed convex subset of $X$. Suppose that 
\begin{compactenum}[$(i)$]
\item $I-\Phi'$ maps $K$ into $K$;
\item $\Phi$ is bounded below in $K$;
\item $\Phi$ satisfies the Palais-Smale condition in $K$.
\end{compactenum}
Then there exists $u_{0}\in X$ such that $\Phi'(u_{0})=0$ and $\inf_{u\in K} \Phi(u)=\Phi(u_{0})$.
\end{prop}

\section{Sub and super solutions to \eqref{P0}}

\noindent
We begin giving the following lemma which gives some useful estimates on the nonlinearity $f$.
\begin{lem}\label{lem1}
Assume that $f$ satisfies \eqref{f1}. Then there exist $0<\underline{\mu}<\lambda_{1}^{s}<\overline{\mu}$ and $\kappa>0$ such that
\begin{equation}\label{ggrowth}
f(x, t)> \underline{\mu} t-\kappa \mbox{ and } f(x, t)> \overline{\mu} t-\kappa
\end{equation}
for any $x\in \Omega$ and $t\in \R$.
\end{lem}
\begin{proof}
From the first inequality in \eqref{f1} we can find $0<\underline{\mu}<\lambda_{1}^{s}$ and $\kappa_{1}>0$ such that $f(x, t)>\underline{\mu} t-\kappa_{1}$ for any $x\in \Omega$ and $t\leq 0$. The second inequality in \eqref{f1} implies that there exist $\overline{\mu}>\lambda_{1}^{s}$ and $\kappa_{2}>0$ such that $f(x, t)>\overline{\mu} t-\kappa_{2}$ for all $x\in \Omega$ and $t\geq 0$. Setting $\kappa=\max\{\kappa_{1}, \kappa_{2}\}$, we can see that \eqref{ggrowth} holds. 

\end{proof}

\begin{remark}\label{REM3.1}
From \eqref{ggrowth} we can deduce that there exists $\tau\in \R$ (independent of $h$) such that for any $t>\tau$ the problem \eqref{P} has no solutions.
Indeed, if $U$ is a solution to \eqref{P}, and we take $\textup{ext}(\varphi_{1})(x, y)=\varphi_{1}(x) \theta(\sqrt{\lambda_{1}}y)$  as test function in the weak formulation of \eqref{P}, we can see that
$$
\lambda_{1}^{s} \int_{\Omega} u \varphi_{1} \, dx=\iint_{\C} y^{1-2s} \nabla U \nabla \textup{ext}(\varphi_{1})\, dx dy=\int_{\Omega} f(x, u)\varphi_{1}\, dx+t,
$$
where we have used $|\varphi_{1}|_{2}=1$ and $\int_{\Omega} h \varphi_{1}\, dx=0$.
Then, if $\int_{\Omega} u\varphi_{1}\, dx>0$, using \eqref{ggrowth} we get
$$
(\lambda^{s}_{1}-\overline{\mu})\int_{\Omega} u \varphi_{1} \, dx+\kappa \int_{\Omega} \varphi_{1}\, dx\geq  t\rightarrow \infty \mbox{ as } t\rightarrow \infty.
$$
A similar estimate can be obtained when $\int_{\Omega} u\varphi_{1}\, dx\leq 0$.
\end{remark}

\noindent
Now, we prove a suitable variant of the weak maximum principle.
\begin{lem}\label{lem2}
Let $\alpha\in L^{\infty}(\Omega)$ such that $ess\sup_{x\in \Omega} \alpha(x)<\lambda_{1}^{s}$, and $U\in \X$ be a supersolution to 
\begin{equation*}
\left\{
\begin{array}{ll}
-\dive(y^{1-2s} \nabla U)= 0 &\mbox{ in } \C\\
U=0 &\mbox{ on } \partial_{L} \C\\
\frac{\partial U}{\partial \nu^{1-2s}}=\alpha(x)u &\mbox{ on } \partial^{0} \C,
\end{array}
\right.
\end{equation*}
that is 
\begin{align}\label{wSs}
\iint_{\C} y^{1-2s}\nabla U \nabla \Psi \, dx dy\geq \int_{\Omega}  \alpha(x) u \psi \, dx
\end{align}
for any $\Psi\in \X$ such that $\Psi \geq 0$.
Then $U\geq 0$ in $\C$. 
\end{lem}
\begin{proof}
Take $\Psi=U^{-}$ in \eqref{wSs}, where $U^{-}=-\min\{U, 0\}\geq 0$. Then, using the following fractional Poincar\'e inequality 
\begin{equation}\label{Poinineq}
\lambda_{1}^{s} |u|_{2}^{2}\leq \|U\|^{2} \mbox{ for any } U\in \X,
\end{equation}
(where we used Theorem \ref{tracethm} and $|u|_{\h}^{2}=\sum_{k\geq 1}\lambda_{k}^{s}|c_{k}|^{2}\geq \lambda_{1}^{s}\sum_{k\geq 1} |c_{k}|^{2}=\lambda_{1}^{s}|u|_{2}^{2}$, since $\lambda_{k}\geq \lambda_{1}$ for all $k\geq 1$)
we get
\begin{align*}
&-\lambda_{1}^{s}\int_{\Omega} (u^{-})^{2} \, dx\geq -\iint_{\C}y^{1-2s} |\nabla U^{-}|^{2}  \, dx dy \\
&\geq-\int_{\Omega} \alpha(x) (u^{-})^{2} \, dx> -\lambda_{1}^{s} \int_{\Omega} (u^{-})^{2} \, dx
\end{align*}
which implies that $U^{-}=0$, that is $U=U^{+}\geq 0$ in $\C$. 
\end{proof}
\begin{remark}
Under the same assumptions on $\alpha(x)$, if $U\in \X$ satisfies 
\begin{align}\label{wss}
\iint_{\C} y^{1-2s}\nabla U \nabla \Psi \, dx dy\leq \int_{\Omega}  \alpha(x) u \psi \, dx
\end{align}
for any $\Psi\in \X$ such that $\Psi\geq 0$, then $U\leq 0$. In fact, it is enough to take $\Psi=U^{+}$ as test function in \eqref{wss}.
\end{remark}

\noindent
At this point we can show that for any $t\in \R$, it is possible to find a subsolution to \eqref{P}.
\begin{lem}\label{lem3}
Assume that \eqref{f1} holds. Then, for any $t\in \R$, the problem \eqref{P} has a subsolution $V_{t}$ which bounds from below  every supersolution to \eqref{P}. Moreover, $V_{t}\in C^{0, \beta}(\bar{\C})$.
\end{lem}
\begin{proof}
Let $M_{t}=\sup_{x\in \bar{\Omega}} |t \varphi_{1}(x)+h(x)|$.  Then, it is clear that the following linear problem
\begin{equation*}
\left\{
\begin{array}{ll}
-\dive(y^{1-2s} \nabla U)=0  &\mbox{ in } \C\\
U=0 &\mbox{ on } \partial_{L} \C \\
\frac{\partial U}{\partial \nu^{1-2s}}=\underline{\mu} u-\kappa-M_{t} &\mbox{ on } \partial^{0} \C
\end{array}
\right.
\end{equation*}
with $\underline{\mu}$ and $C$ as in \eqref{ggrowth}, admits a unique solution $V_{t}\in \X$. In view of Lemma \ref{boundlem}, we can see that $V_{t}\in C^{0, \beta}(\bar{\C})$. 
In the light of \eqref{ggrowth} and the definition of $M_{t}$, we can see that $V_{t}$ is a subsolution to \eqref{P}.
Indeed, for any $\Psi\in \X$ such that $\Psi\geq 0$, we can see that 
\begin{align*}
\iint_{\C} y^{1-2s}\nabla V_{t} \nabla \Psi \, dx dy=\int_{\Omega}  [\underline{\mu} v_{t}-\kappa-M_{t}] \psi \, dx\leq \int_{\Omega}  [f(x, v_{t})+t \varphi_{1}+h] \psi \, dx.
\end{align*}
Now, let $W_{t}$ be a  supersolution to \eqref{P} and we show that $V_{t}\leq W_{t}$. Take $\Psi\in \X$ such that $\Psi\geq 0$ in $\C$.
Then, using that $V_{t}$ and $W_{t}$ are subsolution and supersolution to \eqref{P} respectively, we can see that 
\begin{align*}
&\iint_{\C} y^{1-2s}\nabla (W_{t}-V_{t}) \nabla \Psi \, dx dy \\
&\geq \int_{\Omega}  [f(x, W_{t})+t \varphi_{1}+h-\underline{\mu} v_{t}+\kappa+M_{t}] \psi \, dx \\
&\geq \int_{\Omega}  \underline{\mu}(w_{t}-v_{t}) \psi \, dx.
\end{align*}
Since $0<\underline{\mu}<\lambda_{1}^{s}$, we can apply Lemma \ref{lem2} to deduce that $V_{t}\leq W_{t}$.

\end{proof}

\noindent
In the next lemma we show the existence of a supersolution to \eqref{P} provided that $t$ is sufficiently small.
\begin{lem}\label{lem4}
There exists $t_{0}\in \R$, $t_{0}<0$ such that for any $t\leq t_{0}$ the problem \eqref{P} admits a   supersolution $W_{t}$.
\end{lem}
\begin{proof}
Let $\beta>0$ and $M>0$ such that $M>\max\{|f(x, t)|+\|h\|_{\infty}: x\in \Omega, t\in [0, \beta]\}$.
Let $\Omega''\Subset \Omega'\Subset \Omega$ such that $|\Omega\setminus \Omega''|< \e$, and we define a smooth function $\psi$ such that $\psi=0$ in $\Omega''$, $0\leq \psi\leq M$ in $\Omega$, and $\psi=M$ in $\Omega \setminus \Omega'$. 
Clearly, $\psi\in L^{q}(\Omega)$ for any $q\in [1, \infty]$. 
Now, let $W$ be the unique solution to 
\begin{equation}\label{Vweak0}
\left\{
\begin{array}{ll}
-\dive(y^{1-2s} \nabla W)=0  &\mbox{ in } \C\\
W=0 &\mbox{ on } \partial_{L} \C \\
\frac{\partial W}{\partial \nu^{1-2s}}=\psi &\mbox{ on } \partial^{0} \C. 
\end{array}
\right.
\end{equation}
Choosing $-W^{-}\leq 0$ as test function in \eqref{Vweak0}, and recalling that $\psi\geq 0$ in $\Omega$, we see that 
$$
\iint_{\C}y^{1-2s} |\nabla W^{-}|^{2} \, dx dy=-\int_{\Omega} \psi w^{-} \, dx\leq 0,
$$
that is $W\geq 0$ in $\C$. Moreover, arguing as in Lemma \ref{lem3}, we  deduce that $W\in C^{0, \gamma}(\bar{\C})$, for some $\gamma>0$.
Take $p>N/2s$. We note that the trace $w$  of $W$ is a solution to
\begin{equation*}
\left\{
\begin{array}{ll}
(-\Delta_{\D})^{s} w=\psi  &\mbox{ in } \Omega\\
w=0 &\mbox{ on } \partial \Omega. 
\end{array}
\right.
\end{equation*}
Now, we show that there exist $C_{1}, C_{2}>0$ such that
$$
|w|_{\infty}\leq C_{1}|\psi|_{p}\leq C_{2} \e^{1/p}.
$$
Take $W_{k}=(|W|-k)^{+}sign(W)\in \X$, with $k\geq 0$, as test function in \eqref{Vweak0} and we see that 
\begin{align*}
\|W_{k}\|^{2}=\iint_{\C} y^{1-2s} \nabla W \nabla W_{k} dx dy=\int_{\Omega} \psi w_{k} dx\leq |\psi|_{p}|w_{k}|_{2^{*}_{s}}|\chi_{A_{k}}|_{p_{1}}
\end{align*}
where $A_{k}=\{x\in \Omega: |w(x)|\geq k\}$ and $p_{1}>1$ is such that $\frac{1}{p_{1}}=1-\frac{1}{2^{*}_{s}}-\frac{1}{p}$. Let us note that $p_{1}<2^{*}_{s}$ being $p>\frac{N}{2s}$. From Theorem \ref{Sembedding} it follows that for all $k\geq 0$
\begin{align*}
C_{*}|w_{k}|_{2^{*}_{s}}^{2}\leq |\psi|_{p}|w_{k}|_{2^{*}_{s}}|\chi_{A_{k}}|_{p_{1}},
\end{align*}
which implies that 
\begin{align}\label{KS}
C_{*}|w_{k}|_{2^{*}_{s}}\leq |\psi|_{p} |\chi_{A_{k}}|_{p_{1}}.
\end{align}
Let $h>k$ and note that $A_{h}\subset A_{k}$ and $|w_{k}|\geq h-k$ on $A_{h}$. Using this fact, \eqref{KS} and applying H\"older inequality on the right hand side in \eqref{KS} we can deduce that for all $h>k$
\begin{align*}
|\chi_{A_{h}}|_{2^{*}_{s}}\leq C(h-k)^{-1}|\psi|_{p} |\chi_{A_{k}}|_{2^{*}_{s}}^{\frac{2^{*}_{s}}{p_{1}}},
\end{align*}
where $C$ is independent of $h$ and $k$.
Since $\frac{2^{*}_{s}}{p_{1}}>1$ and applying Lemma B.1 in \cite{KS} with $\varphi(h)=|\chi_{A_{h}}|_{2^{*}_{s}}$, $k_{0}=0$, $\alpha=1$ and $\beta=\frac{2^{*}_{s}}{p_{1}}$ we can see that $|\chi_{A_{h_{0}}}|_{2^{*}_{s}}=0$ with $h_{0}=C|\psi|_{p}(2|\Omega|^{\frac{1}{2^{*}_{s}}})^{\frac{2^{*}_{s}}{2^{*}_{s}-p_{1}}}$. This implies that there exists $C_{1}>0$ such that $|w|_{\infty}\leq C_{1}|\psi|_{p}$. This together with the definition of $\psi$ gives the desired inequalities.
Choosing $\e>0$ such that $C_{2}\e^{1/p}\leq \beta$, we get
\begin{align}\label{bound1}
f(x, w)+h\leq M \mbox{ in } \Omega.
\end{align}
Set $t_{0}=-\frac{M}{\delta_{0}}<0$, where $\delta_{0}=\min\{\varphi_{1}(x): x\in \Omega'\}>0$, and we show that 
\begin{align}\label{bound2}
t_{0} \varphi_{1}+M\leq \psi \mbox{ in } \Omega.
\end{align}
Indeed, if $x\in \Omega'$ then we have 
$$
t_{0} \varphi_{1}+M=-\frac{M}{\delta_{0}}\varphi_{1}+M\leq 0\leq \psi
$$
and if $x\in \Omega\setminus \Omega'$ we obtain
$$
t_{0} \varphi_{1}+M<M=\psi.
$$
Putting together \eqref{bound1} and \eqref{bound2}, we can see that for any $t\leq t_{0}$ it holds
$$
\psi\geq t_{0} \varphi_{1}+M\geq t\varphi_{1}+f(x, w)+h.
$$
This is enough to deduce that $W_{t}=W$ is a supersolution to \eqref{P} for any $t\leq t_{0}$. 

\end{proof}

\noindent
In what follows, we show that $\Phi$ satisfies the Palais-Smale condition. We start proving the following auxiliary lemma.
\begin{lem}\label{lem6.5}
Assume that $f$ satisfies the first assertion in \eqref{f1}.
Let $(U_{n})$ be a sequence in $\X$ such that 
\begin{equation}\label{6.18}
\left | \iint_{\C} y^{1-2s}\nabla U_{n} \nabla V \, dx dy- \int_{\Omega}[ f(x, u_{n}) +t \varphi_{1}+h] v \, dx\right| \leq \e_{n} \|V\|\quad \forall V\in \X
\end{equation}
where $\e_{n}\rightarrow 0$. Then there exists a constant $M>0$ such that $\|U_{n}^{-}\|\leq M$. 
\end{lem}

\begin{proof}
From Lemma \ref{lem1}, we know that 
\begin{equation}\label{6.19}
f(x, t)>\underline{\mu} t- \kappa_{1}\quad \mbox{ for } t\leq 0. 
\end{equation}
Replacing $V$ by $U_{n}^{-}$ in \eqref{6.18} we obtain the following estimate
\begin{equation*}
\iint_{\C} y^{1-2s}|\nabla U_{n}^{-}|^{2} \, dx dy\leq -\int_{\Omega} [f(x, u_{n})+t\varphi_{1}+h] u_{n}^{-} \, dx+ \e_{n} \|U_{n}^{-}\|. 
\end{equation*}
As a consequence, taking into account \eqref{6.19} and recalling that $\varphi_{1}, h\in L^{\infty}(\Omega)$, we deduce
\begin{equation*}
\iint_{\C} y^{1-2s}|\nabla U_{n}^{-}|^{2} \, dx dy \leq \underline{\mu} \int_{\Omega} (u_{n}^{-})^{2} \, dx+ c \int_{\Omega} u_{n}^{-} + \e_{n}\|U_{n}^{-}\|. 
\end{equation*}
Using Poincar\`e inequality \eqref{Poinineq} and Theorem \ref{Sembedding}, we can see that
\begin{equation*}
\left(1-\frac{\underline{\mu}}{ \lambda_{1}^{s}}\right)  \|U_{n}^{-}\|^{2} \leq  c_{1} \|U_{n}^{-}\| + \e_{n}\|U_{n}^{-}\|. 
\end{equation*}
Since $0<\underline{\mu}<\lambda_{1}^{s}$, we can conclude the proof of Lemma \ref{lem6.5}.  
\end{proof}

\begin{lem}\label{lem6.6}
Assume that $f$ satisfies $f$ satisfies \eqref{f1} and \eqref{f2}. 
Then $\Phi$ satisfies $(PS)$ condition. 
\end{lem}

\begin{proof}
Let $(U_{n})$ be a sequence in $\X$ such that 
\begin{align}
&|\Phi(U_{n})|\leq C \label{6.10}\\
&|\langle \Phi'(U_{n}), V\rangle|\leq \e_{n}\|V\| \quad \forall V\in \X \label{6.11},
\end{align}
where $\e_{n}\rightarrow 0$. We begin proving that $(U_{n})$ is bounded.
Assume by contradiction that, up to a subsequence, $(U_{n})$ satisfies 
\begin{equation}\label{6.12}
\|U_{n}\|\rightarrow \infty.
\end{equation}
Let $V_{n}= \frac{U_{n}}{\|U_{n}\|}$. Since $(V_{n})$ is bounded in $\X$, in view of Theorem \ref{Sembedding}, up to a subsequence, we may assume that $V_{n}\rightharpoonup V_{0}$ in $\X$, $v_{n}\rightarrow v_{0}$ in $L^{2}(\Omega)$ and a.e. in $\Omega$. Then there exists  $h_{1}\in L^{2}(\Omega)$ such that $|v_{n}(x)|\leq h_{1}(x)$. It follows from Lemma \ref{lem6.5} that $V_{n}^{-}\rightarrow 0$ in $\X$ and we may assume that $v_{n}^{-}\rightarrow 0$ a.e. in $\Omega$. As a consequence, $v_{0}\geq 0$ in $\Omega$. 

Now, we prove that
\begin{equation}\label{DF1}
g_{n}\equiv \chi_{n} \frac{f(x, u_{n})}{\|U_{n}\|}\rightarrow 0 \quad \mbox{ in } L^{2}(\Omega),
\end{equation}
where $\chi_{n}$ is the characteristic function of the set $\{x\in \Omega: u_{n}(x)\leq 0\}$. Indeed, using \eqref{f2}, we can see that 
$$
|g_{n}|\leq c_{1}h_{1} + \frac{c_{2}}{\|U_{n}\|}.
$$
Since 
\begin{equation}\label{6.21}
|g_{n}|\leq c_{1} \frac{u_{n}^{-}}{\|U_{n}\|}+ \frac{c_{2}}{\|U_{n}\|}\rightarrow 0 \quad \mbox{ a.e. in } \Omega
\end{equation}
we can apply the Dominated Convergence Theorem to deduce that \eqref{DF1} holds.\\
On the other hand, by \eqref{DF1}, we can see that the sequence
\begin{equation}\label{6.22}
\gamma_{n}\equiv (1- \chi_{n}) \frac{f(x, u_{n})}{\|U_{n}\|}\rightharpoonup \gamma \quad \mbox{ in } L^{2}(\Omega)
\end{equation}
for some $\gamma\in L^{2}(\Omega)$ such that $\gamma \geq 0$. To prove this, we first note that \eqref{f2} yields
\begin{equation*}
|\gamma_{n}|\leq c_{1}h_{1} + \frac{c_{2}}{\|U_{n}\|} \leq c_{1}h_{1} +1\in L^{2}(\Omega). 
\end{equation*}
The positiveness of $\gamma$ comes from the following consideration. From the second inequality in \eqref{f1} there exists $r_{0}>0$ such that $f(x, t)\geq 0$ for $t\geq r_{0}$. Let $\xi_{n}$ be the characteristic function of the set $\{x\in \Omega : u_{n}(x) \geq r_{0}\}$. Clearly $\xi_{n}\gamma_{n}\rightharpoonup \gamma$ in $L^{2}(\Omega)$. Then $\gamma\geq 0$ because $\xi_{n}\gamma_{n}$ belongs to the cone of non-negative functions of $L^{2}(\Omega)$, which is a closed and convex set. 

Now, we go back to \eqref{6.11}. Dividing both members by $\|U_{n}\|$ and passing to the limit  as $n\rightarrow \infty$, we obtain
\begin{equation}\label{6.23}
\iint_{\C} y^{1-2s}\nabla V_{0} \nabla V \, dx dy - \int_{\Omega} \gamma v\, dx=0 \quad \forall V\in \X. 
\end{equation}
Here we have used \eqref{6.12}, \eqref{6.21}, \eqref{6.22} and $h, \varphi_{1}\in L^{\infty}(\Omega)$.
It follows from Lemma \ref{lem1} that there is $c>0$ such that $f(x, t)\geq \overline{\mu} t-c$ for $x\in \Omega$ and $t\geq 0$. Therefore $\gamma_{n}\geq \overline{\mu} v_{n}^{+}- c\|U_{n}\|^{-1}$, and taking the limit as $n\rightarrow \infty$, we deduce $\gamma \geq \overline{\mu} v_{0}$. In view of \eqref{6.23} with $V= \textup{ext}(\varphi_{1})(x, y)(=\varphi_{1}(x) \theta(\sqrt{\lambda_{1}}y))$, we obtain 
\begin{equation*}
\lambda_{1}^{s}\int v_{0}\varphi_{1} \, dx= \iint_{\C} y^{1-2s}\nabla V_{0} \nabla \textup{ext}(\varphi_{1}) \, dx dy= \int_{\Omega} \gamma \varphi_{1} \, dx\geq \overline{\mu} \int_{\Omega} v_{0}\varphi_{1} \, dx. 
\end{equation*}
Since $\overline{\mu} >\lambda_{1}^{s}$, we can conclude that $v_{0}\equiv 0$. This and \eqref{6.23} give $\gamma \equiv 0$. \\
Finally, using \eqref{6.11} with $V= V_{n}$, and dividing both members by $\|U_{n}\|$, we get
\begin{equation*}
\left| \iint_{\C} y^{1-2s}|\nabla V_{n}|^{2} \, dx dy- \int_{\Omega} \frac{[f(x, u_{n})+t\varphi_{1}+h]}{\|U_{n}\|}v_{n} \right| \leq \e_{n} \|U_{n}\|^{-1}. 
\end{equation*}
This gives a contradiction because the term on the left hand side goes to $1$ and the right hand side converges to $0$.
Therefore, $(U_{n})$ is bounded in $\X$, and in view of Theorem \ref{Sembedding}, we may assume that 
\begin{align}\label{compconv}
U_{n}\rightharpoonup U \mbox{ in } \X, \mbox{ and }
u_{n}\rightarrow u \mbox{ in } L^{q}(\Omega) \quad \forall q\in [1, 2^{*}).
\end{align}
In view of \eqref{6.11}, we can see that
\begin{equation}\label{6.6}
\left| \iint_{\C} y^{1-2s}\nabla U_{n} \nabla V\, dx dy - \int_{\Omega} [f(x, u_{n})+t\varphi_{1}+h] v \, dx \right| \leq \e_{n} \|V\|
\end{equation}
for any $V\in \X$. Taking $V=U_{n}-U$ and using the continuity of the Nemytskii operator $f(x, u)$ (by \eqref{f2}) and \eqref{compconv}, we can see that \eqref{6.6} yields
$$
\iint_{\C} y^{1-2s}\nabla U_{n} \nabla (U_{n}-U)\, dx dy\rightarrow 0.
$$
Hence $\|U_{n}\|\rightarrow \|U\|$, and recalling that $U_{n}\rightharpoonup U$ in $\X$ we deduce that $U_{n}\rightarrow U$ in $\X$.
\end{proof}

\begin{lem}\label{PSsl}
Assume that $f$ satisfies \eqref{f1}, \eqref{f3} and \eqref{f4}. Then $\Phi$ satisfies the Palais-Smale condition.
\end{lem}
\begin{proof}
In view of Lemma \ref{lem6.5}, we know that $\|U_{n}^{-}\|\leq c$. From the first statement in \eqref{f1}, we can find $0<\mu<\lambda_{1}^{s}$ and $c>0$ such that 
$$
F(x, t)\leq \frac{\mu}{2}t^{2}-c t  \mbox{ for any } x\in \Omega, t\leq 0.
$$ 
As a consequence
\begin{equation}\label{gabba}
\int_{\Omega} F(x, -u_{n}^{-})\, dx \leq \frac{\mu}{2}\int_{\Omega}(u_{n}^{-})^{2}\, dx-c\int_{\Omega}u_{n}^{-}\, dx\leq C.
\end{equation}
Taking into account \eqref{6.10}, \eqref{gabba}, $\|U_{n}^{-}\|\leq c$, $\varphi_{1}, h\in L^{\infty}(\Omega)$, Theorem \ref{tracethm} and Theorem \ref{Sembedding}, we have
\begin{equation}\label{6.25}
\frac{1}{2} \iint_{\C} y^{1-2s} |\nabla U_{n}^{+}|^{2} \, dx-\int_{\Omega} [F(x, u_{n}^{+})+t \varphi_{1} u_{n}^{+}+h u_{n}^{+}]\, dx\leq C.
\end{equation}
Using \eqref{6.11} with $V=U_{n}^{+}$ we get
\begin{equation}\label{6.26}
\left|\iint_{\C}y^{1-2s} |\nabla U_{n}^{+}|^{2} \, dx-\int_{\Omega} [f(x, u_{n}^{+})u_{n}^{+}+t\varphi_{1} u_{n}^{+}+h u_{n}^{+}] \,dx\right|\leq \e_{n} \|U_{n}^{+}\|.
\end{equation}
Then multiplying \eqref{6.25} by $\theta$ and subtracting \eqref{6.26}, we can infer that
\begin{align*}
\left(\frac{\theta}{2}-1\right)  \|U_{n}^{+}\|^{2} &\leq \int_{\Omega} [\theta F(x, u_{n}^{+})-f(x, u_{n}^{+})u_{n}^{+}] \, dx+c|u_{n}^{+}|_{1}+ \e_{n} \|U_{n}^{+}\|+C \\
&\leq C+C\|U_{n}^{+}\|+\e_{n}\|U_{n}^{+}\|
\end{align*}
where we have used \eqref{f4}, Theorem \ref{tracethm} and Theorem \ref{Sembedding}.
Since $\theta>2$, we deduce that $\|U_{n}^{+}\|\leq C$. Then $(U_{n})$ is bounded in $\X$, and we can proceed as in the last part of the proof of Lemma \ref{lem6.6} (here the continuity of the Nemytskii operator $f(x, u)$ is due to \eqref{f3} and Theorem \ref{Sembedding}) to obtain a convergent subsequence.

\end{proof}

\section{proof of theorem \ref{thm1}}
This section is devoted to the proof of the first main result of this work.
From now on we assume that $t\leq t_{0}$, where $t_{0}$ is given by Lemma \ref{lem4}. By Lemma \ref{lem3} and Lemma \ref{lem4} we 
can see that 
\begin{equation}\label{ter}
K=\{U\in \X: v_{t}\leq u\leq w_{t} \mbox{ in } \Omega\}
\end{equation}
is a not empty closed and convex subset of $\X$.

\begin{proof}[Proof of Theorem \ref{thm1}]
In order to find a first solution to \eqref{P}, we aim to apply Proposition \ref{Lm} to the functional $\Phi$ defined in \eqref{phi} and $K$  defined as in \eqref{ter}. 
We begin proving that $\Phi$ is bounded below in $K$. Indeed, if $U\in K$, then $|u(x)|\leq |v_{t}(x)|+|w_{t}(x)|$ and this implies that $|u|_{2^{*}}\leq C$. Taking into account \eqref{f2} and \eqref{f3}, we deduce that $\Phi(U)\geq \frac{1}{2} \|U\|^{2}-C$, that is $(ii)$ holds. In view of Lemma \ref{lem6.6} and Lemma \ref{PSsl}, we also know that $(iii)$ is satisfied. In what follows, we show that the condition $(i)$ is verified changing the norm in $\X$ as follows. \\  
Take $M>0$ such that the function 
\begin{equation}\label{6.38}
\xi\mapsto g(x, \xi):= f(x, \xi) + M\xi + t \varphi_{1}(x)+ h(x)
\end{equation}
is increasing in $\xi\in [a, b]$ (for each $x\in \overline{\Omega}$ fixed), where $a= \min v_{t}$ and $b= \max w_{t}$. 

In virtue of \eqref{Poinineq}, we can see that the norm in $\X$ defined as
\begin{equation*}
\|U\|_{e}^{2} = \iint_{\C} y^{1-2s} |\nabla U|^{2}\,dxdy + M\int_{\Omega} u^{2} \,dx
\end{equation*}
is equivalent to the standard norm $\|\cdot\|$. Let us denote by $\langle \cdot, \cdot\rangle_{e}$ the inner product in $\X$ corresponding to $\|\cdot\|_{e}$. Then we can rewrite \eqref{P} as
\begin{equation}\label{6.39}
\left\{
\begin{array}{ll}
-\dive(y^{1-2s} \nabla U)= 0 &\mbox{ in } \C\\
U=0 &\mbox{ on } \partial_{L} \C \\
\frac{\partial U}{\partial \nu^{1-2s}}=-M u+g(x, u) &\mbox{ on } \partial^{0} \C.
\end{array}
\right.
\end{equation}
The functional associated to \eqref{6.39} is given by
\begin{equation*}
\Psi(U)= \frac{1}{2}\|U\|^{2}_{e} - \int_{\Omega} G(x, u)\, dx, 
\end{equation*}
where $G(x, t)= \int_{0}^{t} g(x, \xi) \, d\xi$.

Let us note that $\Psi=\Phi$ on $K$, $\Psi\in C^{1}(\X, \R)$, $\Psi$ satisfies $(PS)$ and it has the same critical points as the original functional $\Phi$. 
Now we show that $\Psi$ verifies the condition $(i)$ of Proposition \ref{Lm}. 
Let $U\in K$ and let $V= (I- \Psi')U$. This means that $V$ satisfies
\begin{equation*}
\langle V, \varGamma \rangle_{e}= \langle U, \varGamma \rangle_{e}-\langle U, \varGamma \rangle_{e}+ \int_{\Omega} g(x, u)\gamma \, dx
\end{equation*}
for all $\varGamma \in \X$. Then, recalling that $V_{t}$ is a subsolution to \eqref{P}, we deduce that 
\begin{equation*}
\langle V- V_{t}, \varGamma \rangle_{e}\geq \int_{\Omega} [g(x, u)- g(x, v_{t})] \gamma, \quad \forall \varGamma  \in \X, \varGamma \geq 0
\end{equation*}
Since $U\in K$ and $g(\cdot, t)$ is increasing in $t$, we deduce $g(x, u)- g(x, v_{t})\geq 0$ in $\Omega$. Then, $V-V_{t}$ satisfies 
$$
\iint_{\C} y^{1-2s} \nabla (V-V_{t}) \nabla \varGamma  \, dx dy\geq -M\int_{\Omega} (v-v_{t}) \gamma \, dx \quad  \mbox{ for all }  \varGamma  \in \X, \varGamma \geq 0.
$$
Since $-M<0<\lambda_{1}^{s}$, we can apply Lemma \ref{lem2} to get $V\geq V_{t}$. In similar fashion, we can prove that $V\leq W_{t}$ in $\C$. Therefore $V\in K$. Hence, using Proposition \ref{Lm}, we  can find a critical point $U_{0}\in K$ of $\Phi$ such that 
\begin{equation}\label{absurd}
\Psi(U_{0})=\inf_{U\in K} \Psi(U).
\end{equation}
We point out that this proposition ensures only that $U_{0}$ is a minimum of $\Phi$ restricted to $K$.

Now, we aim to apply Proposition \ref{MP} to deduce the existence of a second solution to \eqref{P}. In order to achieve our goal, we begin proving that $U_{0}$ is indeed a local minimum. Let us note that $U_{0}\in \X$ is a solution to \eqref{P}, so we get $U_{0}\in C^{0, \beta}(\bar{\C})$ for some $\beta\in (0, 1)$, by Lemma \ref{boundlem}. \\
We argue by contradiction, and we assume  that $U_{0}$ is not a local minimum of $\Phi$ in $\X$. Thus, for every $\varepsilon>0$ there exists $U_{\e}\in B_{\e}\equiv \overline{B_{\e}(U_{0})}=\{U\in \X: \|U-U_{0}\|\leq \e\}$ such that $\Psi(U_{\e})<\Psi(U_{0})$. Let us consider the functional $\Phi$ restricted to $B_{\e}$, and using Theorem \ref{thm3.1}, there exist $Z_{\e}\in B_{\e}$ and $\lambda_{\e}\leq 0$ such that 
\begin{align}
&\Psi(Z_{\e})= \inf_{B_{\e}} \Psi \leq \Psi(U_{\e})< \Psi(U_{0}) \label{6.40}\\
&\Psi'(Z_{\e})= \lambda_{\e}(Z_{\e}- U_{0}). \label{6.41}
\end{align} 
Let us observe that \eqref{6.41} means 
\begin{equation*}
\langle Z_{\e}, \varGamma \rangle_{e} - \int_{\Omega} g(x, z_{\e}) \gamma \, dx= \lambda_{\e} \langle Z_{\e}- U_{0}, \varGamma \rangle_{e} \quad \forall \varGamma \in \X,
\end{equation*}
or equivalently
\begin{equation}\label{6.43}
\langle Z_{\e}-U_{0}, \varGamma \rangle_{e} =\frac{1}{1- \lambda_{\e}} \int_{\Omega} [g(x, z_{\e}) - g(x, u_{0})]\gamma \, dx\quad \forall \varGamma \in \X. 
\end{equation}
Clearly, from $U_{0}\in C^{0, \beta}(\bar{\C})$ and arguing as in Lemma \ref{boundlem} we get $Z_{\e}\in C^{0, \gamma}(\bar{\C})$ for some $\gamma\in (0,1)$, and $Z_{\e}\rightarrow U_{0}$ in $\X$ as $\e\rightarrow 0$. Then, using $(1-\lambda_{\e})^{-1}\in (0, 1]$, \eqref{f2} (or \eqref{f3}), \eqref{6.43}, $\|Z_{\e}-U_{0}\|\leq \e$ and Lemma \ref{boundlem}, we can see that there exists $C_{1}>0$ independent of $\e$ such that 
$$
|z_{\e}-u_{0}|_{\infty}\leq C_{1}.
$$ 
Therefore we deduce that 
$$
(1-\lambda_{\e})^{-1}  |(g(\cdot, z_{\e}) - g(\cdot, u_{0}))|_{\infty}\leq C_{2},
$$
for some $C_{2}>0$ independent of $\e$,
and using Theorem $1.5$ (1) in \cite{CStinga} (see also Corollary $4.8$ in \cite{BrCDPS}) we have 
$$
|z_{\e}-u_{0}|_{C^{0, \alpha}(\bar{\Omega})}\leq C_{3} \mbox{ for some } \alpha\in (0, 1).
$$ 
Hence, from the Ascoli-Arzel\'a Theorem we can see that $z_{\e}\rightarrow u_{0}$ uniformly in $\bar{\Omega}$ as $\varepsilon \rightarrow 0$. \\
On the other hand, in view of Lemma \ref{lem3} and Lemma \ref{lem4}, we know that $V_{t}$ and $W_{t}$ are H\"older continuous up to the boundary and that they are not solutions to \eqref{P}. 
Then, from Lemma \ref{MPcdds}, we get $V_{t}<U_{0}$ in any compact of $\Omega\times [0, \delta)$, and so $v_{t}< u_{0}$ in $\Omega$. This fact and the above uniform convergence, implies that $v_{t}< z_{\e}$ in $\Omega$, provided that $\e$ is sufficiently small. A similar argument yields $z_{\e}< w_{t}$ in $\Omega$ for any $\e$ small enough. Thus $Z_{\e} \in K$ for small $\e$, and using \eqref{absurd} and \eqref{6.40} we get a contradiction. Hence $U_{0}$ is a local minimum of $\Phi$ in $\X$. 

Finally, we prove that there exists a function $Z\in \X$ such that $\Phi(Z)<\Phi(U_{0})$.
Let us note that the extension $\textup{ext}(\varphi_{1})(x, y)=\varphi_{1}(x) \theta(\sqrt{\lambda_{1}}y)$ of $\varphi_{1}$, is such that $\|\textup{ext}(\varphi_{1})\|^{2}=\lambda_{1}^{s}$, so from \eqref{ggrowth} it holds
$$
\Phi(R\, \textup{ext}(\varphi_{1}))\leq \frac{R^{2}}{2}(\lambda_{1}^{s}-\overline{\mu})+CR+C, \mbox{ for all } R>0.
$$
Recalling that $\overline{\mu}>\lambda_{1}^{s}$, we can choose $R>0$ sufficiently large such that $\Phi(R \, \textup{ext}(\varphi_{1}))<\Phi(U_{0})$. As a consequence, we can apply Proposition \ref{MP} to find a second solution $U_{1}\neq U_{0}$ to \eqref{P}.
\end{proof}


\section{Fractional periodic Ambrosetti-Prodi type problem}
It is easy to see that the results established in the previous sections can be extended when we consider a periodic version of the problem \eqref{P0}, that is
\begin{equation*}
\left\{
\begin{array}{ll}
(-\Delta+m^{2})^{s} u = f(x, u)+t+h \, &\mbox{ in } \, (-\pi, \pi)^{N},\\
u(x+2\pi e_{i})=u(x) \, &\mbox{ for all } \, x\in\R^{N}, i=1, \dots, N
\end{array}
\right.
\end{equation*}
where $m>0$, $s\in (0, 1)$, $N>2s$ and $h\in L^{\infty}((-\pi, \pi)^{N})$ is such that $\int_{(-\pi, \pi)^{N}} h \, dx=0$. \\
Here $f: \R^{N}\times \R \rightarrow \R$ is a locally Lipschitz function, $2\pi$-periodic in $x$, such that
\begin{equation*}
\limsup_{t\rightarrow -\infty} \frac{f(x, t)}{t}<m^{2s}<\liminf_{t\rightarrow \infty} \frac{f(x, t)}{t} \quad \mbox{ uniformly for a.e. } x\in (-\pi, \pi)^{N},
\end{equation*}
and verifying either \eqref{f2} or \eqref{f3} and \eqref{f4}.\\
Let us denote by $\Y$ the completion of
\begin{align*}
\mathcal{C}_{2\pi}^{\infty}(\overline{\R^{N+1}_{+}}):=\Bigl\{&U\in \mathcal{C}^{\infty}(\overline{\R^{N+1}_{+}}): U(x+2\pi e_{i},y)=U(x,y) \\
&\mbox{ for every } (x,y)\in \overline{\R_{+}^{N+1}}, i=1, \dots, N \Bigr\}
\end{align*}
under the $H^{1}(\mathcal{S}_{2\pi},y^{1-2s})$-norm
\begin{equation*}
\|U\|_{\Y}:=\sqrt{\iint_{\mathcal{S}_{2\pi}} y^{1-2s} (|\nabla U|^{2}+m^{2}U^{2}) \, dxdy} \,.
\end{equation*}
It is worth recalling the following fundamental results \cite{A1, A2} concerning the spaces $\Y$ and $\h$.  
\begin{thm}\cite{A1, A2}\label{tracethmper}
There exists a surjective linear operator $\textup{Tr} : \Y\rightarrow \h$  such that:
\begin{itemize}
\item[$(i)$] $\textup{Tr}(U)=U|_{\partial^{0} \mathcal{S}_{2\pi}}$ for all $U\in \mathcal{C}_{2\pi}^{\infty}(\overline{\R^{N+1}_{+}}) \cap \Y$;
\item[$(ii)$] $\textup{Tr}$ is bounded and
\begin{equation}\label{tracein}
\sqrt{\kappa_{s}} |\textup{Tr}(U)|_{\h}\leq \|U\|_{\Y},
\end{equation}
for every $U\in \Y$.
In particular, equality holds in \eqref{tracein} for some $V\in \Y$ if and only if $V$ weakly solves the following equation
$$
-\dive(y^{1-2s} \nabla U)+m^{2}y^{1-2s}U =0 \, \mbox{ in } \, \mathcal{S}_{2\pi}.
$$
\end{itemize}
\end{thm}
\begin{thm}\cite{A1,A2}\label{compactembedding}
Let $N> 2s$. Then  $\textup{Tr}(\Y)$ is continuously embedded in $L^{q}(-\pi,\pi)^{N}$ for any  $1\leq q \leq 2^{*}_{s}$.  Moreover,  $\textup{Tr}(\Y)$ is compactly embedded in $L^{q}(-\pi,\pi)^{N}$  for any  $1\leq q < 2^{*}_{s}$.
\end{thm}

\noindent 
As stated in the below result,  the extension technique \cite{CS} works also in periodic setting (see also \cite{RS, ST}).
\begin{thm}\cite{A1, A2}\label{CSper}
Let $u\in \h$. Then, there exists a unique $U\in \Y$ such that
\begin{equation*}
\left\{
\begin{array}{ll}
-\dive(y^{1-2s} \nabla U)+m^{2}y^{1-2s}U =0 &\mbox{ in }\mathcal{S}_{2\pi}  \\
U_{| {\{x_{i}=-\pi\}}}= U_{| {\{x_{i}=\pi\}}} & \mbox{ on } \partial_{L}\mathcal{S}_{2\pi} \\
U(\cdot, 0)=u  &\mbox{ on } \partial^{0}\mathcal{S}_{2\pi}
\end{array}
\right.
\end{equation*}
and
\begin{align*}
-\lim_{y \rightarrow 0^{+}} y^{1-2s}\frac{\partial U}{\partial y}(x,y)=\kappa_{s} (-\Delta+m^{2})^{s}u(x) \mbox{ in } \mathbb{H}^{-s}_{2\pi}.
\end{align*}
\end{thm}
In view of Theorem \ref{CSper}, we can deduce that the study of \eqref{Pp} is equivalent to consider the following degenerate elliptic Neumann problem
\begin{equation}\label{R}
\left\{
\begin{array}{ll}
-\dive(y^{1-2s} U)+m^{2}y^{1-2s}U= 0 &\mbox{ in } \T:=\B\times (0, \infty)\\
U_{| {\{x_{i}=-\pi\}}}= U_{| {\{x_{i}=\pi\}}} &\mbox{ on } \partial_{L} \T:=\partial \B \times [0, \infty) \\
\frac{\partial U}{\partial \nu^{1-2s}}=\kappa_{s}[f(x, u)+t +h] &\mbox{ on } \partial^{0} \T:=\B\times \{0\}.
\end{array}
\right.
\end{equation}

\noindent
Solutions to \eqref{R} can be obtained as critical points of the functional  $\Phi: \Y\rightarrow \R$ defined as
\begin{equation*}
\Phi(U)=\frac{1}{2} \|U\|^{2}_{\Y}-\int_{\B} [F(x, u)+(t +h)u] \,dx.
\end{equation*}
From Theorem \ref{tracethmper}, we can see that the following useful inequalities hold  
\begin{equation}\label{RAF}
\kappa_{s}m^{2s}|u|^{2}_{2}\leq \kappa_{s}|u|_{\h}\leq \|U\|^{2}_{\Y}  \mbox{ for any } U\in \Y.
\end{equation}
Thus, with slight modifications, we can prove that Lemma \ref{lem1}, Lemma \ref{lem2}, Lemma \ref{lem3}, Lemma \ref{lem6.5}, Lemma \ref{lem6.6} and Lemma \ref{PSsl} hold with $\Omega=(-\pi, \pi)^{N}$, $\lambda_{1}=m^{2}$, $\varphi_{1}=\frac{1}{(2\pi)^{N}}$, \eqref{RAF} instead of \eqref{Poinineq}, and replacing Theorem \ref{tracethm} by Theorem \ref{tracethmper}, and $\|\cdot\|$ by $\|\cdot\|_{\Y}$. Concerning the proof of Lemma \ref{lem4}, the existence of a supersolution to \eqref{R} for $t$ sufficiently small, can be easily obtained considering the unique solution $W$ of the following problem
\begin{equation*}
\left\{
\begin{array}{ll}
-\dive(y^{1-2s} W)+m^{2}y^{1-2s}W= 0 &\mbox{ in } \T:=\B\times (0, \infty)\\
W_{| {\{x_{i}=-\pi\}}}= W_{| {\{x_{i}=\pi\}}} &\mbox{ on } \partial_{L} \T:=\partial \B \times [0, \infty) \\
\frac{\partial W}{\partial \nu^{1-2s}}=\kappa_{s}[f_{1}+h] &\mbox{ on } \partial^{0} \T:=\B\times \{0\}
\end{array}
\right.
\end{equation*}
where $f_{1}(x)=f(x, 0)-\-int_{\B} f(\tau, 0)\, d\tau$. Then $W$ is a supersolution to \eqref{R} provided that
$$
f_{1}+h \geq f(x, w) + t+h = f(x, w) - f(x, 0) + f(x, 0) + t+ h, 
$$
that is
$$
t\leq -\left[\-int_{\B} f(\tau, 0)\, d\tau + \max_{x\in [-\pi, \pi]^{N}} [f(x, w)-f(x, 0)]\right] =:t_{0}.
$$
Therefore, taking into account that a suitable variant of Lemma \ref{boundlem} holds in our framework (see Theorem $9$ in \cite{A2} and Proposition $3$ in \cite{FF}), and making use of the Harnack inequality (see Proposition $2$ in \cite{FF} and Theorem $2.1$ in \cite{RS}), we can argue as in Section $4$ to deduce that \eqref{R} admits at least two periodic solutions for any $t\leq t_{0}$. This ends the proof of Theorem \ref{thm2}.

\smallskip


\end{document}